\documentclass{amsart}

\usepackage{a4wide,latexsym,amssymb,amsthm,amsmath,color}

\theoremstyle{plain}

\newtheorem{theorem}{Theorem}
\newtheorem{lemma}{Lemma}

\newtheorem{corollary}{Corollary}

\theoremstyle{remark}

\newtheorem{remark}{Remark}

\def\R{\mathbb{R}}

\newcommand{\nks}{\ensuremath{\mathbb{S}^3 \times \mathbb{S}^3}} 

\begin{document}

\title[Lagrangian Submanifolds with Constant Angle Functions ]{Lagrangian Submanifolds with Constant Angle Functions of the nearly K{\" a}hler $\mathbb{S}^3\times\mathbb{S}^3$}

\author[B. Bekta\c s]{Burcu Bekta\c s}
\address{Istanbul Technical University, Faculty of Science and Letters, Department of Mathematics, 34469, Maslak, 
Istanbul, Turkey} 
\email{bektasbu@itu.edu.tr}

\author[M. Moruz]{Marilena Moruz}
\address{LAMAV, ISTV2 Universit\'e de Valenciennes Campus du Mont Houy 59313 Valenciennes Cedex 9 France}
\email{marilena.moruz@gmail.com}

\author[J. Van der Veken]{Joeri Van der Veken}
\address{KU Leuven, Department of Mathematics, Celestijnenlaan 200B -- Box 2400, BE-3001 Leuven, Belgium} 
\email{joeri.vanderveken@wis.kuleuven.be}

\author[L. Vrancken]{Luc Vrancken}
\address{LAMAV, ISTV2 \\Universit\'e de Valenciennes\\ Campus du Mont Houy\\ 59313 Valenciennes Cedex 9\\ France and KU  Leuven \\ Department of Mathematics \\Celestijnenlaan 200B -- Box 2400 \\ BE-3001 Leuven \\ Belgium}  
\email{luc.vrancken@univ-valenciennes.fr}

\thanks{This work was partially supported by the project 3E160361 (Lagrangian and calibrated submanifolds) of the KU Leuven research fund and part of it was carried out while the first author visited KU Leuven, supported by The Scientific and Technological Research Council of Turkey (TUBITAK), under grant 1059B141500244.}

\begin{abstract}
We study Lagrangian submanifolds of the nearly K\"ahler $\nks$ with respect to their, so called, angle functions. We show that if all angle functions are constant, then the submanifold is either totally geodesic or has constant sectional curvature and  there is a classification theorem that follows from \cite{dvw}. Moreover, we show that if precisely one angle function is constant, then it must be equal to $0,\frac{\pi}{3}$ or $\frac{2\pi}{3}$. Using then two remarkable constructions together with the classification of Lagrangian submanifolds of which the first component has nowhere maximal rank from \cite{eu}, we obtain a classification of such Lagrangian submanifolds.
\end{abstract}

\keywords{}

\maketitle

%%%%%%%%%%%%%%%%%%%%%%%%%%%%%%%%%%%%%%%%%%%%%%%%%%%%%%%%%%%%%%%%%%%%%%%%%%%%%%%%%%%%%%%%%%%%%%%%%%%%%%%%%%%%%%%%%%%%%%%

\section{Introduction}

The systematic study of nearly K\"ahler manifolds originates with the work of Gray (\cite{9}). In this paper we continue the study of Lagrangian submanifolds of the nearly K\"ahler $\nks$ started in \cite{ss}, \cite{zhd}, \cite{dvw} and \cite{eu}. The homogeneous nearly K\"ahler $\nks$ is one of the four homogeneous nearly K\"ahler manifolds as was shown by Butruille (\cite{b}). Note that only recently the first complete non homogeneous examples were discovered in \cite{Haskins}.\\
A submanifold of  a $6$-dimensional nearly K\"ahler manifold is called Lagrangian if it is $3$-dimensional and the almost complex structure maps the tangent space into the normal space. Such submanifolds are especially interesting (see \cite{Ivanov}, \cite{ss}), as they are necessarily minimal and orientable. The Lagrangian submanifolds of the nearly K\"ahler $\mathbb{S}^6$ were treated in \cite{Dillen-Vrancken-Verstr}, \cite{Dillen-Vrancken}, \cite{Ejiri2}, \cite{Lotay}, \cite{Palmer}, \cite{Vrancken1}, \cite{Vrancken}.  The first examples appeared in \cite{ms} and \cite{ss}. An important tool in the study of Lagrangian submanifolds of the nearly K\"ahler $\nks$ are the so called angle functions $\theta_1,\theta_2,\theta_3$, which were introduced in \cite{dvw} (see also section \ref{sec2} of the present paper). Note that $\theta_1+\theta_2+\theta_3$ must always be an integer multiple of $\pi$. In this paper we show the following theorems.
\begin{theorem}
Let $f:M\rightarrow \nks: x\mapsto f(x)=(p(x),q(x))$ be a Lagrangian immersion. If the angle functions $\theta_1,\theta_2,\theta_3$ are constant, then either $M$ is totally geodesic or $M$ has constant sectional curvature.
\end{theorem}
A complete classification, as stated in Corollary \ref{cc}, follows immediately by applying the main results of \cite{zhd} and \cite{dvw}. Note that as for an equivariant Lagrangian immersion the angle functions are constant, this yields at the same time a classification of all the equivariant Lagrangian submanifolds. Next, we also deal with the case that precisely one of the angle functions, which we may assume to be $\theta_1$, is constant. In that case, we obtain:
\begin{theorem}
Let $f:M\rightarrow \nks: x\mapsto f(x)=(p(x),q(x))$ be a Lagrangian immersion, which is neither totally geodesic nor of constant sectional curvature. Assume that $\theta_1$ is constant. Then up to a multiple of $\pi$, we have that either $\theta_1=0,\frac{\pi}{3},\frac{2\pi}{3}$.
\end{theorem}
Note that the case $\theta_1=\frac{\pi}{3}$ corresponds precisely with those immersions for which the first component has nowhere maximal rank, see \cite{eu}, Theorem 1. Such immersions are related to minimal surfaces in $\mathbb{S}^3$ and a complete classification is obtained in \cite{eu}, see also Theorems 5, 6, 7 and 8 in Section \ref{results} of the present paper. Similarly, the case that $\theta_1=\frac{2\pi}{3}$ corresponds to the case that the second component has nowhere maximal rank. In order to obtain a classification in these last two cases and as a tool in several of the other theorems, we introduce two constructions denoted by  $\,\tilde{}$ and ${}^*$, which allows us to construct from a Lagrangian immersion $f:M\rightarrow\nks$, two new Lagrangian immersions $\tilde{f}:M\rightarrow\nks$ and $f^*:M\rightarrow\nks$, which all induce the same metric on $M$ (see Section \ref{results}).

\section{Preliminaries} \label{sec2}

A nearly K{\"a}hler manifold is an almost Hermitian manifold with almost complex structure $J$ such that 
$\tilde\nabla J$ is skew symmetric, i.e.,
\begin{equation}
(\tilde\nabla_XJ)Y+(\tilde\nabla_YJ)X=0
\end{equation} 
for all tangent vector $X,Y$ and for $\tilde\nabla$ the Levi-Civita connection on the manifold.\\ 
A submanifold of an almost Hermitian manifold is called Lagrangian if the almost complex structure $J$
interchanges the tangent and the normal spaces and if the dimension of the submanifold is half the dimension of the ambient manifold.
 
\subsection{The Nearly K{\"a}hler structure of $\mathbb{S}^3\times\mathbb{S}^3$}
In this section, we recall the homogeneous nearly K{\"a}hler structure of $\mathbb{S}^3\times\mathbb{S}^3$ and mention 
some  known results from \cite{dvw} and \cite{zhd}.\\
First, we can identify the $3$-sphere $\mathbb{S}^3$ with the set of all the unit quaternions in $\mathbb{H}$, i.e., 
$$
\mathbb{S}^3=\{p\in\mathbb{H}\;|\;\langle p,p\rangle=1\},
$$ 
where the metric $\langle \cdot , \cdot\rangle$ is induced from the metric on $\mathbb{R}^4$. 
Let $ i, j,k$ denote the standard imaginary unit quaternions. Then the vector fields $X_1,X_2,X_3$  given by
\begin{equation*}
\begin{split}
 X_1(p)&=p\,i=(-x_2,x_1,x_4,-x_3),\\
 X_2(p)&=p\, j=(-x_3,-x_4,x_1,x_2),\\
 X_3(p)&=-p\,k=(x_4,-x_3,x_2,-x_1),
 \end{split}
\end{equation*}
where $p=x_1+x_2\,i+x_3\,j+x_4\,k\in
\mathbb{S}^3$, form a basis of the tangent bundle $T\mathbb{S}^3$. 
Hence, the tangent space of $\mathbb{S}^3$ is defined by 
$T_p\mathbb{S}^3=\{p\alpha\;|\; \alpha\in\mbox{Im}{\mathbb{H}}\}$. 

Let $Z_{(p,q)}$ be a tangent vector of $\mathbb{S}^3\times\mathbb{S}^3$ at $(p,q)$. From the known natural
identification $T_{(p,q)}(\mathbb{S}^3\times\mathbb{S}^3)\cong T_p\mathbb{S}^3\oplus T_q\mathbb{S}^3$, 
we write $Z_{(p,q)}=(U_{(p,q)}, V_{(p,q)})$ or simply $Z=(U,V)$.
Now, we define the vector fields on $\mathbb{S}^3\times\mathbb{S}^3$ as
\begin{align*}
 \tilde E_1(p,q) &= (p\,i,0),   &   \tilde F_1(p,q) &= (0,q\,i),\\
 \tilde E_2(p,q) &= (p\,j,0),   & \tilde F_2(p,q) &= (0,q\,j),\\
 \tilde E_3(p,q) &= -(p\,k,0),  &   \tilde F_3(p,q) &= -(0,q\,k),
\end{align*}
which are mutually orthonogal with respect to the usual Euclidean product metric on $\mathbb{S}^3\times\mathbb{S}^3$.
The Lie brackets are $[\tilde E_i,\tilde E_j]=-2\varepsilon_{ijk}\tilde E_k$, $[\tilde F_i, \tilde F_j]=-2\varepsilon_{ijk}\tilde F_k$ and $[\tilde E_i,\tilde F_j]=0$,
where 
$$\varepsilon_{ijk}=\left\{
\begin{aligned}
& 1, \qquad\,  {\rm if}\ (ijk)\  {\rm is\ an\ even\ permutation\ of\ (123)},\\
&-1,\quad  {\rm if}\ (ijk)\ {\rm is\ an\ odd\ permutation\ of\ (123)},\\
& 0, \ \qquad {\rm if\ otherwise.}
\end{aligned}
\right.
$$
The almost complex structure $J$ on $\mathbb{S}^3\times\mathbb{S}^3$ is defined by
\begin{equation}
\label{nkalcomp.}
J(U,V)_{(p,q)}=\frac{1}{\sqrt{3}}\left(2pq^{-1}V-U, -2qp^{-1}U+V\right),
\end{equation}
for $(U,V)\in T_{(p,q)}(\mathbb{S}^3\times\mathbb{S}^3)$ (see \cite{b}).
The nearly K\"ahler metric on $\mathbb{S}^3\times\mathbb{S}^3$ with which we choose to work is the Hermitian metric associated to the usual Euclidean product
metric on $\mathbb{S}^3\times\mathbb{S}^3$:
\begin{align}
\label{kahlermetric}
\begin{split}
   g(Z,Z') &= \frac{1}{2} \left(\langle Z,Z'\rangle + \langle JZ,JZ'\rangle\right)\\
          &= \frac{4}{3} \left(\langle U,U'\rangle +  \langle V,V'\rangle\right)
             -\frac{2}{3} \left(\langle p^{-1}U,q^{-1}V'\rangle +  \langle p^{-1}U',q^{-1}V\rangle\right),
\end{split}
\end{align}
where $Z=(U,V)$ and $Z'=(U',V')$. In the first line $\langle\cdot,\cdot\rangle$ stands for the usual Euclidean product metric on
$\mathbb{S}^3\times \mathbb{S}^3$ and in the second line $\langle \cdot,\cdot\rangle$ stands for the usual Euclidean metric on 
$\mathbb{S}^3$. From the definition, it can be seen that the almost complex structure $J$ is compatible with the metric $g$.

\begin{lemma}
(\cite{bddv}
\label{lem:levicivita}
The Levi-Civita connection $\tilde \nabla$ on $\mathbb{S}^3\times\mathbb{S}^3$ with respect to the metric $g$ is given by
\begin{align*}
\tilde\nabla_{\tilde E_i}\tilde E_j &= -\varepsilon_{ijk}\tilde E_k, &\tilde\nabla_{\tilde E_i}\tilde F_j &= \frac{\varepsilon_{ijk}}{3}(\tilde E_k -\tilde F_k),\\
\tilde\nabla_{\tilde F_i}\tilde E_j &= \frac{\varepsilon_{ijk}}{3} (\tilde F_k -\tilde E_k), & \tilde\nabla_{\tilde F_i}\tilde F_j &= -\varepsilon_{ijk} \tilde F_k.
\end{align*}
\end{lemma}
Then, we obtain the following:
\begin{equation}
\label{eq:G}
\begin{split}
(\tilde\nabla_{\tilde E_i} J)\tilde E_j &= -\frac{2}{3\sqrt{3}}\varepsilon_{ijk}(\tilde E_k + 2\tilde F_k), \quad
\;(\tilde\nabla_{\tilde E_i} J)\tilde F_j = -\frac{2}{3\sqrt{3}}\varepsilon_{ijk} (\tilde E_k - \tilde F_k), \\
(\tilde\nabla_{\tilde F_i} J)\tilde E_j &= -\frac{2}{3\sqrt{3}}\varepsilon_{ijk}(\tilde E_k -\tilde F_k), \quad \;
(\tilde\nabla_{\tilde F_i} J)\tilde F_j =  -\frac{2}{3\sqrt{3}}\varepsilon_{ijk} (2\tilde E_k +\tilde F_k).
\end{split}
\end{equation}
Let $G:=\tilde\nabla J$. Then $G$ is skew-symmetric and it satisfies
\begin{equation}
G(X,JY)=-JG(X,Y), \quad g(G(X,Y),Z)+g(G(X,Z),Y)=0,
\end{equation}
for any vectors fields $X,Y,Z$ tangent to $\mathbb{S}^3\times\mathbb{S}^3$.
Therefore, $\mathbb{S}^3\times\mathbb{S}^3$ equipped with $g$ and the almost complex structure $J$, becomes a nearly K\"ahler manifold.\\
Moreover, we introduce the almost product structure $P$, defined in \cite{bddv} as
\begin{equation}
\label{nkalprod.}
P(U,V)_{(p,q)}=(pq^{-1}V,qp^{-1}U)
\end{equation}
for $(U,V)\in T_{(p,q)}(\mathbb{S}^3\times\mathbb{S}^3)$.
It satisfies the following properties:
\begin{align*}
P^2&=Id,\; \text{i.e. $P$ is involutive,}\\
PJ&=-JP,\;\text{ i.e. $P$ and $J$ anti-commute},\\ 
g(PZ,PZ')&=g(Z,Z'), \ \text{i.e. $P$ is compatible with $g$,}\\
g(PZ,Z')&=g(Z,PZ'),\ \text{ i.e. $P$ is symmetric.} 
\end{align*}

\subsection{Lagrangian Submanifolds of the nearly K{\"a}hler $\mathbb{S}^3\times\mathbb{S}^3$.}
Assume that $M$ is a Lagrangian submanifold in the nearly K{\"a}hler $\mathbb{S}^3 \times \mathbb{S}^3$. 
Since $M$ is Lagrangian, the pull-back of $T(\mathbb{S}^3 \times \mathbb{S}^3)$ to $M$ splits into $TM\oplus JTM$. 
Hence, there are two endomorphisms $A,B\colon TM\to TM$ such that the restriction $P|_{TM}$ of $P$ to
the submanifold equals $A+ JB$, that is $PX = AX + JBX$ for all $X\in TM$. 
This formula together with the fact that $P$ and $J$ anticommute determine $P$ on the normal space
by $PJX = -JPX = BX- JAX$. The endomorphisms $A$ and $B$ have the following properties given in \cite{dvw}:
\begin{itemize}
\item  $A$ and $B$ are symmetric operators which satisfy $A^2+B^2=\mbox{Id}$;
\item The covariant derivatives of the endomorphisms $A$ and $B$ are 
\begin{align}
 (\nabla_X A)Y = B S_{JX} Y - Jh(X,BY) +\frac{1}{2}\left(JG(X,AY)-AJG(X,Y)\right),\label{opA}\\
 (\nabla_X B)Y =  Jh(X,AY) -A S_{JX} Y +\frac{1}{2}\left(JG(X,BY)-BJG(X,Y)\right),\label{opB}
\end{align}
where $\nabla$ is the induced connection on $M$, $h$ is the second fundamental form and $S$ is the shape operator of the Lagrangian immersion;
\item The Lie brackets of $A$ and $B$ vanish, that is, $A$ and $B$ can be diagonalized  simultaneously at a point $p$ of $M$.
\end{itemize}
Therefore, at each point $p\in M$ there is an orthonormal basis $e_1$, $e_2$, $e_3\in T_p M$ such that
\begin{equation}\label{P}
 Pe_i = \cos (2\theta_i) e_i + \sin (2\theta_i) Je_i,~\forall~i=1,2,3.
\end{equation}
Now we extend the orthonormal basis $e_1$, $e_2$, $e_3$ at a point $p$ to a frame on a
neighborhood of $p$ in the Lagrangian submanifold. By Lemma~1.1-1.2 in \cite{szabo} the orthonormal basis at a point
can be extended to a differentiable frame $E_1,E_2, E_3$ on an open dense neighborhood
where the multiplicities of the eigenvalues of $A$ and $B$ are constant.
Taking also into account the properties of $G$ we know that there exists a local
orthonormal frame $\{E_1,E_2,E_3\}$ on an open dense subset of $M$
such that
\begin{equation}
\label{diagop}
AE_i=\cos (2\theta_i) E_i, \quad BE_i=\sin (2\theta_i) E_i,\quad  JG(E_i,E_j)=\frac{1}{\sqrt{3}}\varepsilon_{ijk}E_k.
\end{equation}
Notice that, in a general sense, for an immersion  $f:M\rightarrow \nks$ there exist $A, B:TM\rightarrow TM$ with  eigenvectors $E_i$ and corresponding angle functions $\theta_i$ such that, on the image of $M$ we may write by \eqref{diagop} and \eqref{P}:
\begin{equation}\label{Ponimage}
Pdf(E_i)=df(AE_i)+Jdf(BE_i) \quad \Leftrightarrow \quad Pdf(E_i)=\cos(2\theta_i) df(E_i)+\sin(2\theta_i)Jdf(E_i),
\end{equation}for $i=1,2,3.$\\
The equations of Gauss and Codazzi, respectively, state that
\begin{equation}
\label{Gauss2}
 \begin{split}
   R(X,Y)Z &= \frac{5}{12}\left(g(Y,Z)X - g(X,Z)Y\right) \\
                 &\quad + \frac{1}{3}\left(g(AY,Z)AX - g(AX,Z)AY  +  g(BY,Z)BX - g(BX,Z)BY\right)\\
                 &\quad +[S_{JX},S_{JY}] Z
 \end{split}
\end{equation}
and
\begin{equation}
\label{codazzi}
\begin{split}
\nabla h(X,Y,Z)-&\nabla h (Y,X,Z)=\\
&\frac{1}{3}\left(g(AY,Z)JBX - g(AX,Z)JBY  -  g(BY,Z)JAX + g(BX,Z)JAY\right).
\end{split}
\end{equation}
The functions $\omega_{ij}^k$ are defined by $$\nabla_{E_i}E_j=\omega_{ij}^kE_k$$ and satisfy $\omega_{ij}^k=-\omega_{ik}^j,$ where $\nabla$ is the Levi-Civita connection on $M$. Here we use the Einstein summation convention. 
We denote by $$h_{ij}^k=g(h(E_i, E_j), JE_k)$$ the components of the cubic form $g(h(\cdot,\cdot),J\cdot)$ and we see that $h_{ij}^k$ are totally symmetric on the Lagrangian submanifold.
We recall the following lemmas.
\begin{lemma}
\cite{dvw}
\label{lem:sumzero}
The sum of the angles $\theta_1 + \theta_2  + \theta_3$ is zero modulo $\pi$.
\end{lemma}

\begin{lemma}
\cite{dvw} 
\label{lem:sff}
The derivatives of the angles $\theta_i$ give the components of the second fundamental form
\begin{equation}
\label{derv1}
E_i(\theta_j) = -h_{jj}^i,
\end{equation}
except~$h_{12}^3$. The second fundamental form and covariant derivative are related by
\begin{equation}
\label{derv2}
h_{ij}^k \cos(\theta_j-\theta_k) = (\frac{\sqrt{3}}{6}\varepsilon_{ij}^k-\omega_{ij}^k )\sin(\theta_j-\theta_k).
\end{equation}
\end{lemma}

\begin{lemma}
\cite{dvw}
\label{lem:equalangles}
If two of the angles are equal modulo $\pi$, then the Lagrangian submanifold is totally geodesic.
\end{lemma}
\begin{remark}
\label{sinpi}
By Lemma \ref{lem:equalangles}, we may see that if the Lagrangian submanifold is not totally geodesic, then $\sin(\theta_i-\theta_j)\neq0$,  for $i\neq j$.
\end{remark}

\section{Results}\label{results}
From now on, we identify the tangent vector $X$ with $df(X)$.
\begin{theorem}
\label{thm:comp1}
Let $f:M\longrightarrow\mathbb{S}^3\times\mathbb{S}^3$ be a Lagrangian immersion into the nearly 
K\"ahler manifold $\mathbb{S}^3\times\mathbb{S}^3$, given by
$f=(p,q)$ with angle functions $\theta_1, \theta_2, \theta_3$ and eigenvectors $E_1,E_2,E_3$.
Then $\tilde{f}:M\longrightarrow\mathbb{S}^3\times\mathbb{S}^3$ given by $\tilde{f}=(q,p)$ satisfies:
\begin{itemize}
\item[(i)] $\tilde{f}$ is a Lagrangian immersion,
\item [(ii)] $f$ and $\tilde f$ induce the same metric on $M$,
\item[(iii)] $E_1, E_2,E_3$ are also eigendirections of the operators $\tilde{A},\tilde{B}$ corresponding to the immersion $\tilde{f}$ and the angle functions $\tilde\theta_1, \tilde\theta_2, \tilde\theta_3$ are given by $\tilde\theta_i=\pi-\theta_i,$ for $i=1,2,3$.  
\end{itemize}
\end{theorem}

\begin{proof}
Let  $f:M\longrightarrow\mathbb{S}^3\times\mathbb{S}^3$ given by $f=(p,q)$ 
be a Lagrangian immersion with the angle functions $\theta_1, \theta_2, \theta_3$. 
Then, for any point on $M$, we have a differentiable frame $\{E_1,E_2,E_3\}$ along $M$  satisfying \eqref{Ponimage} such that  
\begin{align}
\label{eq:tangentf}
df(E_i)&=(p\alpha_i,q\beta_i)_{(p,q)}, \; i=1,2,3,
\end{align}
where $\alpha_i, \beta_i$ are imaginary quaternions.
Moreover, for $\tilde{f}$ we have as well
$$
d\tilde{f}(E_i)=(q\beta_i, p\alpha_i)_{(q,p)}, \; i=1,2,3.
$$
 From equations \eqref{nkalcomp.} and \eqref{nkalprod.} 
a direct calculation gives that 
\begin{align}
\label{Pdf}
Pdf(E_i)&=(p\beta_i, q\alpha_i)_{(p,q)},\\
\label{eq:almcomplexf}
Jdf(E_i)&=\frac{1}{\sqrt{3}}\left(p(2\beta_i-\alpha_i), q(-2\alpha_i+\beta_i)\right)_{(p,q)},
\end{align}
and 
\begin{align}
\label{Pdftilde}
Pd\tilde{f}({E}_{i})&=(q\alpha_i, p\beta_i)_{(q,p)},\\
\label{Jdftilde}
Jd\tilde{f}({E}_i)&=\frac{1}{\sqrt{3}}\left(q(2\alpha_i-\beta_i), p(-2\beta_i+\alpha_i)\right)_{(q,p)},
\end{align}
for $i=1,2,3$.
The conditions for $f$ and $\tilde{f}$ to be  Lagrangian immersions write out, respectively, as
\begin{align*}
&g(df(E_i),Jdf(E_j))=0\text{ for $i\neq j$, }\\
&g(d\tilde{f}(E_i),Jd\tilde{f}(E_j))=0 \text{ for $i\neq j$. }
\end{align*}
By \eqref{kahlermetric} and by the previous relations, these conditions become
\begin{align*} 
\frac{4}{3}(\langle\alpha_i,2\beta_j-\alpha_j\rangle+\langle\beta_i,-2\alpha_j+\beta_j \rangle)
-\frac{2}{3}(\langle \alpha_{i} , -2\alpha_j+\beta_j \rangle+\langle 2\beta_j-\alpha_j,\beta_i \rangle)=0,\\
\frac{4}{3}(\langle\beta_i,2\alpha_j-\beta_j\rangle+\langle \alpha_i,-2\beta_j+\alpha_j \rangle)
-\frac{2}{3}(\langle \beta_i,-2\beta_j+\alpha_j \rangle+\langle 2\alpha_j-\beta_j,\alpha_i \rangle)=0,
\end{align*}
respectively.
Since both are equivalent to
$
\langle \alpha_i, \beta_j \rangle -\langle \beta_i,\alpha_j\rangle=0,
$
we conclude that $\tilde{f}$ is a Lagrangian immersion if and only if $f$ is a Lagrangian immersion.\\
In order to prove (ii), we must show that $g(df(E_i),df(E_j))=g(d\tilde{f}(E_i),d\tilde{f}(E_j))$.
By straightforward computations, using \eqref{kahlermetric}, we have
\begin{align}\label{preserve}
g(df(E_i),df(E_j))=&\frac{1}{2}\Bigl( \langle df(E_i),df(E_j)\rangle+ \langle Jdf(E_i),Jdf(E_j)\rangle\Bigr)\\
			=&\frac{1}{2}\Big(  \langle(p\alpha_i,q\beta_i), (p\alpha_j,q\beta_j) \rangle     \Bigr.+\nonumber \\
			&+\Bigl. \frac{1}{3} \langle (p(2\beta_i-\alpha_i),q(-2\alpha_i+\beta_i)  ),(p(2\beta_j-\alpha_j),q(-2\alpha_j+\beta_j))   \rangle \Bigr)\nonumber \\
		=&\frac{2}{3}\Bigl(  2\langle \alpha_i,\alpha_j \rangle +2\langle \beta_i,\beta_j \rangle   - \langle \beta_i,\alpha_j \rangle -\langle \alpha_i,\beta_j \rangle \Bigr).\nonumber
\end{align}
Similarly, we have
\begin{align*}
g(d\tilde f(E_i),d\tilde f(E_j))=&\frac{1}{2}\Bigl( \langle d\tilde f(E_i),d\tilde f(E_j)\rangle+ \langle Jd\tilde f(E_i),Jd\tilde f(E_j)\rangle\Bigr)\\
			=&\frac{1}{2}\Big(  \langle(q\beta_i,p\alpha_i), (q\beta_j,p\alpha_j) \rangle     \Bigr.+\\
			&+\Bigl. \frac{1}{3} \langle (q(2\alpha_i-\beta_i),  p(-2\beta_i+\alpha_i) ),( q(2\alpha_j-\beta_j), p(-2\beta_j+\alpha_j))   \rangle \Bigr)\\
		=&\frac{2}{3}\Bigl(  2\langle \alpha_i,\alpha_j \rangle +2\langle \beta_i,\beta_j \rangle   - \langle \beta_i,\alpha_j \rangle -\langle \alpha_i,\beta_j \rangle \Bigr)
\end{align*}
and we can easily notice that the metric is preserved under the transformation $\tilde{f}$.\\
In order to prove (iii), we see from \eqref{Ponimage} that 
\begin{equation}\label{PPdf}
Pdf(E_i)=\cos(2\theta_i)df(E_i)+\sin(2\theta_i)Jdf(E_i),
\end{equation}
and there exist $\tilde A, \tilde B:TM\rightarrow TM$ with eigenvectors $\tilde E_i$ and angle functions $\tilde\theta_i$ such that 
\begin{equation}
\label{PPdftilde}
Pd\tilde{f}(\tilde E_i)=\cos(2\tilde\theta_i)d\tilde{f}(\tilde E_i)+\sin(2\tilde\theta_i)Jd\tilde{f}(\tilde E_i).
\end{equation}
From \eqref{eq:tangentf} and  \eqref{eq:almcomplexf}, we replace $df(E_i)$ and $Jdf(E_i)$ in \eqref{PPdf} and get:
\begin{equation}
Pdf(E_i)=\Bigl(p\Bigl(\cos(2\theta_i)\alpha_i+\frac{1}{\sqrt{3}}\sin(2\theta_i)(2\beta_i-\alpha_i)\Bigr), q\Bigl(\cos(2\theta_i)\beta_i+\frac{1}{\sqrt{3}}\sin(2\theta_i)(-2\alpha_i+\beta_i)\Bigr)\Bigr).
\end{equation}
Considering now equation \eqref{Pdf} as well, we obtain
\begin{align*}
\alpha_i&=\cos(2\theta_i)\beta_i+\frac{1}{\sqrt{3}}\sin(2\theta_i)(-2\alpha_i+\beta_i),\\
\beta_i&=\cos(2\theta_i)\alpha_i+\frac{1}{\sqrt{3}}\sin(2\theta_i)(2\beta_i-\alpha_i).
\end{align*}
Replacing $\alpha_i$ and $\beta_i$ in \eqref{Pdftilde} with the latter expressions gives
$$
Pd\tilde{f}(E_i)=\cos(-2\theta_i)d\tilde{f}(E_i)+\sin(-2\theta_i)Jd\tilde{f}(E_i).
$$
Comparing this with \eqref{PPdftilde}, we see that  $E_i$ are the eigenvectors of $\tilde{A}$ and $\tilde{B}$ with angle functions
$$
\tilde\theta_i=\pi-\theta_i.
$$

\end{proof}

\begin{theorem}
\label{thm:comp2}
Let $f:M\longrightarrow\mathbb{S}^3\times\mathbb{S}^3$ be a Lagrangian immersion into the nearly 
K\"ahler manifold $\mathbb{S}^3\times\mathbb{S}^3$ given by
$f=(p,q)$ with angle functions $\theta_1, \theta_2, \theta_3$ and eigenvectors $E_1,E_2, E_3.$ 
Then, $f^*:M\longrightarrow\mathbb{S}^3\times\mathbb{S}^3$ given by $f^*=(\bar{p},q\bar{p})$ satisfies:
\begin{itemize}
\item[(i)] $f^*$ is a Lagrangian immersion,
\item [(ii)] $f$ and $f^*$ induce the same metric on $M$,
\item[(iii)] $E_1,E_2,E_3$ are also eigendirections of the operators $A^*,B^*$ corresponding to the immersion $f^*$ and the angle functions $\theta^*_1, \theta^*_2, \theta^*_3$ are given by $\theta^*_i=\frac{2\pi}{3}-\theta_i,$ for $i=1,2,3$.  
\end{itemize}
\end{theorem}

\begin{proof}
Let $f:M\longrightarrow\mathbb{S}^3\times\mathbb{S}^3$ given by $f=(p,q)$ be a Lagrangian immersion with the angle functions $\theta_1, \theta_2, \theta_3$. 
Then, for any point on $M$, we have a differentiable frame $\{E_i\} $ along $M$  satisfying \eqref{Ponimage} and we may write	
\begin{align}
df(E_i)&=(p\alpha_i,q\beta_i)_{(p,q)},\label{df2}\\
df^*(E_i)&=(\bar{p}\alpha^*_i, q\bar{p}\beta^*_i)_{(\bar{p},q\bar p)},\label{dfstar}
\end{align}
for  $ i=1,2,3$ and $\alpha_i, \beta_i,\alpha^*_i, \beta^*_i$  imaginary quaternions.
Moreover  we have that 
\begin{align*}
&\alpha^*_i=-p\alpha_i\bar p,\\
&\beta^*_i=p(\beta_i-\alpha_i)\bar{p}, 
\end{align*}
where we have used
\begin{align*}
df^*(E_i)=D_{E_i}f^*=(D_{E_i}\bar{p},D_{E_i}(q\bar{p}))=(\overline{D_{E_i}p}, (D_{E_i}q)\bar{p}+q(\overline{D_{E_i}p}))\overset{\eqref{df2}}{=}(-\alpha_i\bar p, q(\beta_i-\alpha_i)\bar p)
\end{align*}for the Euclidean connection $D$. 
Furthermore, by \eqref{nkalcomp.} and \eqref{nkalprod.}, we obtain again \eqref{Pdf} and \eqref{eq:almcomplexf} as well as
\begin{align}
\label{Pdftilde2}
Pdf^*({E}_{i})&=((\beta_i-\alpha_i)\bar{p}, -q\alpha_i\bar p  )_{(\bar p,q\bar p)},\\
\label{Jdfstar}
Jdf^*({E}_i)&=\frac{1}{\sqrt{3}}\left((2\beta_i-\alpha_i)\bar p, q(\alpha_i+\beta_i)\bar p \right)_{(\bar p,q\bar p)}
\end{align}
for $i=1,2,3$.
A straightforward computation gives,  for $i\neq j$, that 
\begin{align*}
g(df^*(E_i),Jdf^*(E_j))=\frac{2}{\sqrt{3}}(\langle \beta_i,\alpha_j\rangle-\langle \alpha_i,\beta_j\rangle),
\end{align*}
which, as in the proof of Theorem \ref{thm:comp1}, shows that $f^*$ is a Lagrangian immersion if and only if $f$ is a Lagrangian immersion.\\
To prove (ii), we must show that $g(df(E_i),df(E_j))=g(df^*(E_i),df^*(E_j))$.
By straightforward computations, using \eqref{kahlermetric}, we have
\begin{align*}
g(df^*(E_i),d f^*(E_j))&=\frac{1}{2}\Bigl( \langle df^*(E_i),df^*(E_j)\rangle+ \langle Jdf^*(E_i),Jdf^*(E_j)\rangle\Bigr)\\
			&=\frac{1}{2}\Big(  \langle( -\alpha_i \bar p , q(\beta_i-\alpha_i)\bar p), ( -\alpha_j \bar p , q(\beta_j-\alpha_j)\bar p) \rangle     \Bigr.+\\
			&+\Bigl. \frac{1}{3} \langle ((2\beta_i-\alpha_i)\bar p,  q(\alpha_i+\beta_i)\bar p ),((2\beta_j-\alpha_j)\bar p,  q(\alpha _j+\beta_j)\bar p ) \rangle \Bigr)\\
		&=\frac{2}{3}\Bigl(  2\langle \alpha_i,\alpha_j \rangle +2\langle \beta_i,\beta_j \rangle   - \langle \beta_i,\alpha_j \rangle -\langle \alpha_i,\beta_j \rangle \Bigr),
\end{align*}
and, comparing it to \eqref{preserve}, we can easily notice that the metric is preserved under the transformation $f^*$.\\
In order to prove (iii), we see from \eqref{Ponimage} that 
\begin{equation}\label{PPdf2}
Pdf(E_i)=\cos(2\theta_i)df(E_i)+\sin(2\theta_i)Jdf(E_i),
\end{equation}
and, associated with the second immersion $f^*$, there exist $ A^*,  B^*:TM\rightarrow TM$ with eigenvectors $ E^*_i$ and angle functions $\theta^*_i$ such that 
\begin{equation}
\label{PPdftilde2}
Pdf^*( E^*_i)=\cos(2\theta^*_i)df^*( E^*_i)+\sin(2\theta^*_i)Jdf^*( E^*_i).
\end{equation}
As in the proof of the previous theorem, we have 
\begin{align*}
\alpha_i&=\cos(2\theta_i)\beta_i+\frac{1}{\sqrt{3}}\sin(2\theta_i)(-2\alpha_i+\beta_i),\\
\beta_i&=\cos(2\theta_i)\alpha_i+\frac{1}{\sqrt{3}}\sin(2\theta_i)(2\beta_i-\alpha_i).
\end{align*}
On the one hand, replacing $\alpha_i$ and $\beta_i$ in \eqref{Pdftilde2} with the latter expressions, we see that
$$
Pdf^*(E_i)=\Big( [\cos(2\theta_i)(\alpha_i-\beta_i)+\frac{1}{\sqrt{3}}\sin(2\theta_i)(\beta_i+\alpha_i)]\bar p,-q[\cos(2\theta_i)\beta_i+\frac{1}{\sqrt{3}}\sin(2\theta_i)(-2\alpha_i+\beta_i)  ]\bar p \Big).
$$
On the other hand, we see that for $\theta^*_i=\frac{2\pi}{3}-\theta_i$, the following holds:
\begin{align*}
\cos(2\theta^*_i)df^*( E_i)+&\sin(2\theta^*_i)Jdf^*( E_i)=\cos(\frac{4\pi}{3}-2\theta_i)df^*(E_i)+\sin(\frac{4\pi}{3}-2\theta_i)Jdf^*(E_i)\\
=&\frac{1}{2}[(-\cos(2\theta_i-\sqrt{3}\sin(2\theta_i)))df^*(E_i)+(-\sqrt{3}\cos(2\theta_i)+\sin(2\theta_i)Jdf^*(E_i))]\\
\overset{\eqref{dfstar}, \eqref{Jdfstar}}{=}&([\cos(2\theta_i)(\alpha_i-\beta_i)+\frac{\sin(2\theta_i)}{\sqrt{3}}(\alpha_i+\beta_i)]\bar p, q[ \cos(2\theta_i)(-\beta_i)+\frac{\sin(2\theta_i)}{\sqrt{3}}(2\alpha_i-\beta_i) ]\bar p ).
\end{align*}
Therefore, \eqref{PPdftilde2} holds for $E_i^*=E_i$ and $\theta_i^*=\frac{2\pi}{3}-\theta_i$. This concludes point (iii) of the theorem.

\end{proof}

%%%%%%%%%%

\begin{lemma}
\label{lem:fundeqcons}
Let $M$ be a Lagrangian submanifold of the nearly K\"ahler manifold $ \nks$
with constant angle functions $\theta_1, \theta_2, \theta_3$. 
\begin{itemize}
\item[i.] If $M$ is a non-totally geodesic submanifold, then the nonzero components of $\omega_{ij}^k$ are given by 
\begin{align}
\label{eq:consconn1}
\omega_{12}^3&=\frac{\sqrt{3}}{6}-\frac{\cos{(\theta_2-\theta_3)}}{\sin{(\theta_2-\theta_3)}}h_{12}^3,\\
\label{eq:consconn2}
\omega_{23}^1&=\frac{\sqrt{3}}{6}+\frac{\cos{(\theta_1-\theta_3)}}{\sin{(\theta_1-\theta_3)}}h_{12}^3,\\
\label{eq:consconn3}
\omega_{31}^2&=\frac{\sqrt{3}}{6}-\frac{\cos{(\theta_1-\theta_2)}}{\sin{(\theta_1-\theta_2)}}h_{12}^3.
\end{align}
 
\item[ii.] The Codazzi equations of the submanifold $M$ are as followings:
\begin{align}
\label{eq:conscod1}
E_i(h_{12}^3)&=0, \quad i=1,2,3,\\
\label{eq:conscod2}
h_{12}^3\left(2(\omega_{13}^2+\omega_{21}^3)+\frac{1}{\sqrt{3}}\right)&=\frac{1}{3}\sin(2(\theta_1-\theta_2)),\\
\label{eq:conscod3}
h_{12}^3\left(2(\omega_{12}^3+\omega_{31}^2)-\frac{1}{\sqrt{3}}\right)&=\frac{1}{3}\sin(2(\theta_1-\theta_3)),\\
\label{eq:conscod4}
h_{12}^3\left(2(\omega_{21}^3+\omega_{32}^1)+\frac{1}{\sqrt{3}}\right)&=\frac{1}{3}\sin(2(\theta_2-\theta_3)).
\end{align}

\item[iii.] The Gauss equations of the submanifold $M$ are given by 
\begin{align}
\label{eq:consgauss1}
\frac{5}{12}+\frac{1}{3}\cos(2(\theta_1-\theta_2))-(h_{12}^3)^2=
-\omega_{21}^3\omega_{13}^2+\omega_{12}^3\omega_{31}^2-\omega_{21}^3\omega_{31}^2,\\
\label{eq:consgauss2}
\frac{5}{12}+\frac{1}{3}\cos(2(\theta_1-\theta_3))-(h_{12}^3)^2=
-\omega_{31}^2\omega_{12}^3+\omega_{13}^2\omega_{21}^3-\omega_{31}^2\omega_{21}^3,\\
\label{eq:consgauss3}
\frac{5}{12}+\frac{1}{3}\cos(2(\theta_2-\theta_3))-(h_{12}^3)^2=
-\omega_{32}^1\omega_{21}^3+\omega_{23}^1\omega_{12}^3-\omega_{32}^1\omega_{12}^3.
\end{align} 
\end{itemize}
\end{lemma}

\begin{proof}
Suppose that $M$ is a Lagrangian submanifold of the nearly K{\"a}hler $\mathbb{S}^3\times\mathbb{S}^3$ for which the angle functions
$\theta_1, \theta_2, \theta_3$ are constant. Thus,  equation \eqref{derv1} immediately implies that 
all coefficients of the second fundamental form are zero except $h_{12}^3$. Using \eqref{derv2}, we see that $\omega_{ii}^j=0,i\neq j$. As $\omega_{ij}^k=-\omega_{ik}^j$, it follows that $\omega_{12}^3,\omega_{23}^1,\omega_{31}^2$ are the only non-zero components out of $\omega_{ij}^k$.
From \eqref{derv2} and by Remark \ref{sinpi},  we calculate the nonzero connection forms as in \eqref{eq:consconn1}-\eqref{eq:consconn3}. 
Taking $E_1,E_2,E_3$ and $\ E_3,E_1,E_2$ for the vector fields $X, Y, Z$ in the Codazzi equation \eqref{codazzi}, we get  \eqref{eq:conscod1},  and for $E_1,E_2,E_2; E_1,E_3,E_3; E_2,E_3,E_3$ we obtain \eqref{eq:conscod1}-\eqref{eq:conscod4}. Moreover, we evaluate the Gauss equation \eqref{Gauss2} successively for 
$E_1, E_2, E_2$; $E_1, E_3, E_3$; $E_3, E_2, E_2$ and then we obtain the given equations, respectively.
\end{proof}

\begin{theorem}
\label{thm:clascons}
A Lagrangian submanifold of the nearly K\"ahler manifold $\mathbb{S}^3\times\mathbb{S}^3$ for which all angle functions are constant is either totally geodesic or has constant sectional curvature in $\mathbb{S}^3\times\mathbb{S}^3$.  
\end{theorem}

\begin{proof}
Suppose that $M$ is a Lagrangian submanifold in the nearly K\"ahler $\mathbb{S}^3\times\mathbb{S}^3$ with constant angle functions $\theta_1, \theta_2, \theta_3$. From  equation \eqref{derv1} and the fact that $h_{ij}^k$ are totally symmetric,
all coefficients are zero except $h_{12}^3$. Also, the Codazzi equations given by \eqref{eq:conscod1}-\eqref{eq:conscod4}
are valid for $M$. Equation \eqref{eq:conscod1} implies that $h_{12}^3$ is constant and thus, there are two cases that may occur:

\textit{Case 1.} $h_{12}^3=0$, that is, $M$ is a totally geodesic Lagrangian submanifold in the nearly K\"ahler 
$\mathbb{S}^3\times\mathbb{S}^3$. 

\textit{Case 2.} $h_{12}^3$ is a nonzero constant, that is, $M$ is non-totally geodesic. 
In this case, the nonzero components of $\omega_{ij}^k$ for the submanifold $M$ are given by 
\eqref{eq:consconn1}-\eqref{eq:consconn3}. 
Replacing the coresponding $\omega_{ij}^k$ in the Codazzi equations given by \eqref{eq:conscod2}-\eqref{eq:conscod4}, 
we obtain the following system of equations:
\begin{align}
\label{eq:1}
2(h_{12}^3)^2-\frac{1}{\sqrt{3}}\frac{\sin(\theta_1-\theta_3)\sin(\theta_2-\theta_3)}{\sin(\theta_1-\theta_2)} h_{12}^3
-\frac{2}{3}\cos(\theta_1-\theta_2)\sin(\theta_1-\theta_3)\sin(\theta_2-\theta_3)=0,\\
\label{eq:2}
2(h_{12}^3)^2-\frac{1}{\sqrt{3}}\frac{\sin(\theta_1-\theta_2)\sin(\theta_1-\theta_3)}{\sin(\theta_2-\theta_3)} h_{12}^3
-\frac{2}{3}\cos(\theta_2-\theta_3)\sin(\theta_1-\theta_2)\sin(\theta_1-\theta_3)=0,\\
\label{eq:3}
2(h_{12}^3)^2-\frac{1}{\sqrt{3}}\frac{\sin(\theta_2-\theta_3)\sin(\theta_1-\theta_2)}{\sin(\theta_1-\theta_3)} h_{12}^3
+\frac{2}{3}\cos(\theta_1-\theta_3)\sin(\theta_2-\theta_3)\sin(\theta_1-\theta_2)=0.
\end{align}
Notice that by Remark \ref{sinpi}, we have $\sin(\theta_i-\theta_j)\neq 0$. Considering the above system of equations as a linear system in $2(h_{12}^3)^2, h_{12}^3,1$ and since $h_{12}^3$ is a nonzero constant, we see that the matrix of the system  must have determinant zero. By a direct calculation, we find 
$$\sin(\theta_1+\theta_2-2\theta_3)\sin(\theta_2+\theta_3-2\theta_1)\sin(\theta_1+\theta_3-2\theta_2)=0.$$ 
Given the symmetry in $\theta_1,\theta_2,\theta_3$, it is sufficient to assume that $\sin(\theta_1+\theta_3-2\theta_2)=0$. Thus, considering also Lemma \ref{lem:sumzero}, we may write 
\begin{align}
&\theta_1+\theta_3-2\theta_2=k_1\pi,\\
&\theta_1+\theta_2+\theta_3=k_2\pi
\end{align}for $k_1,k_2\in\mathbb{Z}$.
As each angle function is determined modulo $\pi$, there exist $l_1,l_2,l_3\in\mathbb{Z}$ such that the above equations are satisfied by new angle functions $\theta_1+l_1\pi$, $\theta_2+l_2\pi$  and $\theta_3+l_3\pi$:
\begin{align*}
(\theta_1+l_1\pi)+(\theta_3+l_3\pi)-2(\theta_2+l_2\pi)=k_1^*\pi,\\
(\theta_1+l_1\pi)+(\theta_2+l_2\pi)+(\theta_3+l_3\pi)=k_2^*\pi.
\end{align*}
 This implies that 
\begin{align*}
k_1&=k_1^*-(l_1+l_3-2l_2),\\
k_2&=k_2^*-(l_1+l_2+l_3).
\end{align*}
Hence, we may assume $k_2=1$. Allowing now only changes of angles which preserve this property, we must have that $l_2=-l_1-l_3$ and $k_1=k_1^*-3(l_1+l_3).$ So we may additionally assume that $k_1\in\{-1,0,1\}$. Therefore, we have three cases:
\begin{itemize}
\item [(i)] $\theta_1+\theta_3-2\theta_2=-\pi \ \text{and}\ \theta_1+\theta_2+\theta_3=\pi,$
\item [(ii)]  $\theta_1+\theta_3-2\theta_2=0 \ \text{and}\ \theta_1+\theta_2+\theta_3=\pi,$
\item [(iii)]  $\theta_1+\theta_3-2\theta_2=\pi\ \text{and}\ \theta_1+\theta_2+\theta_3=\pi.$
\end{itemize}
Finally, this reduces to 
\begin{itemize}
\item[(i)] $\theta_2=0$ and $\theta_1+\theta_3=\pi$,
\item[(ii)] $\theta_2=\frac{\pi}{3}$ and $\theta_1+\theta_3=\frac{2\pi}{3}$,
\item[(iii)] $\theta_2=\frac{2\pi}{3}$ and $\theta_1+\theta_3=\frac{\pi}{3}$.
\end{itemize}
Using the relations between the angles $\theta_i$ and $\tilde\theta_i$, $\theta^*_i$ 
of the Lagrangian immersions $\tilde{f}$ and $f^*$ in Theorem \ref{thm:comp1} and Theorem \ref{thm:comp2} respectively, 
these three cases can be reduced to a single case, as  shown below:
\begin{equation}
\begin{array}{l}
\theta_2=\frac{\pi}{3}\\
\theta_1+\theta_3=\frac{2\pi}{3}
\end{array}
\overset{\tilde{f}}{\longleftrightarrow}
\begin{array}{l}
\tilde{\theta}_2=\frac{2\pi}{3}\\
\tilde{\theta}_1+\tilde{\theta}_3=\frac{\pi}{3}
\end{array}
\overset{f^*}{\longleftrightarrow}
\begin{array}{l}
{\theta}_2^*=0\\
{\theta}_1^*+\tilde{\theta}_3^*=\pi.
\end{array}
\end{equation}
Remark that according to Theorems \ref{thm:comp1} and \ref{thm:comp2}, the metric  $g$ given by \eqref{kahlermetric} is preserved under transformations $\tilde{f}, f^*$ from which we deduce that the sectional curvature of $M$ is the same in each case.
Therefore, it is sufficient to consider the case that $\theta_2=\frac{\pi}{3}$ and $\theta_1+\theta_3=\frac{2\pi}{3}$. 
By straightforward computations, equations \eqref{eq:conscod2}-\eqref{eq:consgauss3}  reduce to 
\begin{align}
\label{eq:1n}
&2(h_{12}^3)^2-\frac{1}{\sqrt{3}}\sin{(2\alpha)}\ {h_{12}^3}-\frac{1}{3}\sin^2{(2\alpha)}=0,\\
\label{eq:3n}
&2(h_{12}^3)^2-\frac{1}{\sqrt{3}}\frac{\sin^2{\alpha}}{\sin (2\alpha)}\ {h_{12}^3}
+\frac{2}{3}\sin^2{\alpha}\cos( 2\alpha)=0,
\end{align}
where $\alpha:=\theta_1-\frac{\pi}{3}$.
Solving this system of equations, we see that there are four cases that we must discuss:
\begin{itemize}
\item [(a)] $h_{12}^3=-\frac{1}{2}$ and $\alpha=-\frac{\pi}{3}+k\pi,$
\item [(b)] $h_{12}^3=-\frac{1}{4}$ and $\alpha=\frac{\pi}{3}+k\pi,$
\item [(c)] $h_{12}^3=\frac{1}{4}$ and $\alpha=-\frac{\pi}{3}+k\pi,$
\item [(d)] $h_{12}^3=\frac{1}{2}$ and $\alpha=\frac{\pi}{3}+k\pi$
\end{itemize}for some $k\in \mathbb{Z}.$\\
Remark that cases (c) and (d) reduce to cases (a) and (b), respectively, by changing the basis $\{E_1,E_2,E_3\}$ with $\{E_3,E_2,-E_1\}$. Therefore, we will only consider cases (a) and (b).

\textit{Case (a):} $h_{12}^3=-\frac{1}{2}$ and $\theta_1=0, \theta_2=\frac{\pi}{3}, \theta_3=\frac{2\pi}{3}$.
From \eqref{eq:consconn1}-\eqref{eq:consconn3},  we find that all connection forms are zero. 
Thus, $M$ is a flat Lagrangian submanifold in the nearly K{\"a}hler $\mathbb{S}^3\times\mathbb{S}^3$. 

\textit{Case (b):} $h_{12}^3=-\frac{1}{4}$ and $\theta_1=\frac{2\pi}{3}, \theta_2=\frac{\pi}{3}, \theta_3=0$,
In this case, we have that $\omega_{12}^3=\omega_{23}^1=\omega_{31}^2=\frac{\sqrt{3}}{4}$. 
By a straightforward computation, we find that $M$ has constant sectional curvature which is equal to $\frac{3}{16}$.
As a result, the Lagrangian submanifold $M$ of the nearly K{\"a}hler manifold $\mathbb{S}^3\times\mathbb{S}^3$ with 
constant angle functions $\theta_1, \theta_2, \theta_3$, which is not totally geodesic, has constant sectional curvature.
\end{proof}

Combining the classification theorems in \cite{dvw} and \cite{zhd} and Theorem \ref{thm:clascons}, we state the following:

\begin{corollary}\label{cc}
A Lagrangian submanifold in the nearly K{\"a}hler manifold $\mathbb{S}^3\times\mathbb{S}^3$ whose all angle functions are constant is locally congruent to one of the following immersions: 
\begin{enumerate}
\item
$f\colon \mathbb{S}^3 \to\mathbb{S}^3\times\mathbb{S}^3: u\mapsto (u,1),$ 
\item
$f\colon\mathbb{S}^3 \to\mathbb{S}^3\times\mathbb{S}^3: u\mapsto (1,u),$
\item 
$f\colon \mathbb{S}^3\to\mathbb{S}^3\times\mathbb{S}^3: u\mapsto (u,u),$
\item
$f\colon\mathbb{S}^3\to\mathbb{S}^3\times\mathbb{S}^3: u\mapsto (u,ui),$
\item
$f\colon\mathbb{S}^3\to\mathbb{S}^3\times\mathbb{S}^3: u\mapsto(u^{-1},uiu^{-1}),$
\item 
$f\colon \mathbb{S}^3\to\mathbb{S}^3\times\mathbb{S}^3: u\mapsto(uiu^{-1},u^{-1}),$

\item  $f\colon \mathbb{S}^3\to\mathbb{S}^3\times\mathbb{S}^3: u\mapsto (uiu^{-1},uju^{-1}),$

\item $f: \mathbb R^3\to\mathbb{S}^3\times\mathbb{S}^3:(u,v,w)\mapsto (p(u,w),q(u,v))$, where $p$ and $q$ are constant mean curvature tori in $\mathbb{S}^3$
\begin{align*}
p(u,w)&=\left (\cos u \cos w,\cos u \sin w,\sin u \cos w,\sin u \sin w\right),\\
  q(u,v)&=\frac{1}{\sqrt{2}}\left (\cos v \left(\sin u+\cos u\right),\sin
   v \left(\sin u+\cos u\right)\right . ,\\
   &\qquad \left .  \cos v \left(\sin u-\cos u\right),\sin v \left(\sin
   u-\cos u\right) \right).
\end{align*} 
\end{enumerate}
\end{corollary}

\begin{theorem}
Let $M$ be a Lagrangian submanifold in the nearly K{\"a}hler manifold $\mathbb{S}^3\times\mathbb{S}^3$ with angle functions 
$\theta_1, \theta_2, \theta_3$. If precisely one of the angle functions is constant, then up to a multiple of $\pi$, it can be either $0, \frac{\pi}{3}$  or 
$\frac{2\pi}{3}$. 
\end{theorem}

\begin{proof}
First, we may denote the three angle functions by
\begin{align*}
&2\theta_1=2 c,\\
&2\theta_2= 2\Lambda-c,\\
&2\theta_3= -2\Lambda-c,
\end{align*} where $c\in \R$ and $\Lambda$ is some non constant function. Then, we may write the conditions following from the minimality of the  Lagrangian immersion: 
\begin{align}
\label{minimality}
\begin{array}{l}
h_{11}^1+h_{12}^2+h_{13}^3=0,\\
h_{11}^2+h_{22}^2+h_{23}^3=0,\\
h_{11}^3+h_{22}^3+h_{33}^3=0.
\end{array}
\end{align}
We are going to use the  definitions of $\nabla A$ and $\nabla B$  in \eqref{opA} and \eqref{opB} and then evaluate these relations for different vectors in the basis  in order to get information about the functions $\omega_{ij}^k$ and  $h_{ij}^k$. For $X=Y=E_1$ in \eqref{opA} we obtain that 
\begin{align}
&h_{12}^2=-h_{13}^3,\nonumber\\
&\omega_{11}^2=-\frac{h_{11}^2 (\cos (c-2 \Lambda )+\cos (2 c))}{\sin (c-2 \Lambda )+\sin (2 c)},\label{55}\\
&\omega_{11}^3=-\frac{h_{11}^3 (\sin (2 c)-\sin (2 \Lambda +c))}{\cos (2 c)-\cos (2 \Lambda +c)}.\nonumber
\end{align}
If we take $X=E_1$ and $Y=E_2$ in \eqref{opA} and \eqref{opB}, we see that 
\begin{align}
&E_1(\Lambda)=h_{13}^3,\\
&\omega_{12}^3=\frac{\sqrt{3}}{6}-h_{12}^3\cot{2\Lambda}
\end{align}
and, for $X=E_2$ and $Y=E_1$ in \eqref{opA}, we obtain
\begin{align}
h_{11}^2&=0,\label{h2}\\
\omega_{21}^2&=-\frac{h_{13}^3 (\sin (2 c)-\sin (c-2 \Lambda ))}{\cos (2 c)-\cos (c-2 \Lambda )}\\
\omega_{21}^3&= -h_{12}^3 \cot \left(\Lambda +\frac{3 c}{2}\right)-\frac{\sqrt{3}}{6}.
\end{align}
%%%%%%%%%%%%%%%%%%%%%
Then we choose successively $X=E_3,Y=E_1$, $X=E_2,Y=E_3$ and $X=E_3,Y=E_2$    in relations \eqref{opA} and \eqref{opB}
and obtain
\begin{align}
h_{11}^3&=0,\label{h3}\\
\omega_{31}^2&=\frac{\sqrt{3}}{6} - h_{12}^3 \cot \left(\frac{3 c}{2}-\Lambda \right),\\
\omega_{31}^3&=\frac{h_{13}^3 (\sin (2 c)-\sin (2 \Lambda +c))}{\cos (2 c)-\cos (2 \Lambda +c)},\\
\omega_{22}^3&=-\cot 2\Lambda h_{22}^3,\\
\omega_{32}^3&=-\cot 2\Lambda h_{23}^3\\
E_2(\Lambda)&=h_{23}^3,\label{dere2}\\
E_3(\Lambda)&=-h_{22}^3\label{dere3}.
\end{align}
We can easily see  from \eqref{55}, \eqref{h2} and \eqref{h3} that 
$$
\omega_{11}^2=0\quad \text{and} \quad \omega_{11}^3=0, $$
and if we consider, as well, the relations in \eqref{minimality}, we have that
$$
h_{33}^3=-h_{22}^3,\quad h_{11}^1=0\quad  \text{and} \quad h_{22}^2=-h_{23 }^3.
$$
Next, we are going to use the definition for $\nabla h$ in \eqref{codazzi} and take different values for the vectors $X,Y$ and $Z$. Thus, we evaluate it for $E_1,E_2,E_1$ and $E_1,E_3,E_1$. Looking at the component in $E_3$ of the resulting two vectors, we obtain the following relations, respectively:
\begin{align*}
%\left\{
E_1(h_{12}^3)&=\frac{h_{13}^3}{\sqrt{3}}-2 h_{12}^3 h_{13}^3 \left(\cot (2 \Lambda )+\frac{2 \sin (2 \Lambda )}{\cos (2 \Lambda )-\cos (3 c)}\right), \\
E_1(h_{13}^3)&=\frac{1}{3} \Big(\cot \left(\Lambda +\frac{3 c}{2}\right)\big( 1-\cos (2 \Lambda +3 c)+6
   (h_{12}^3)^2 + 6 (h_{13}^3)^2 )+6 (h_{12}^3)^2 \cot (2 \Lambda )-\sqrt{3} h_{12}^3\Big).
%\right.
\end{align*}
Taking again $X=E_1,Y=E_2,Z=E_1$ in \eqref{codazzi} as just done previously, we look at the component of $E_2$ this time, after replacing  $E_1(h_{13}^3)$ from the above equations, and we get that 
$$
\sin (3 c) \csc \left(\frac{3 c}{2}-\Lambda \right) \csc \left(\Lambda +\frac{3 c}{2}\right) \left(\cos (4 \Lambda
   )-2 \cos (2 \Lambda ) \cos (3 c)+12 (h_{12}^3)^2+12 (h_{13}^3)^2+1\right)=0.
$$
As  $\Lambda$ is not constant, this implies that  $\cos (4 \Lambda
   )-2 \cos (2 \Lambda ) \cos (3 c)+12 (h_{12}^3)^2+12 (h_{13}^3)^2+1=0$ or  $\sin(3c)=0$. \\
\emph{Case 1.} $\sin(3c)=0.$ 
In this case, considering that $\theta_1\in[0,\pi]$, it is straightforward to see that $c\in\{ 0,\frac{\pi}{3},\frac{2\pi}{3}\}$.\\In the following we will show that the other case cannot occur.\\
\hfill\break
\emph{Case 2.} $\sin(3c)\neq 0$.  It follows that  
\begin{equation}\label{zero}
\cos (4 \Lambda
   )-2 \cos (2 \Lambda ) \cos (3 c)+12 (h_{12}^3)^2+12 (h_{13}^3)^2+1=0,
\end{equation}
and, therefore, its derivative with respect to $E_1$ vanishes too:
$$
h_{13}^3 \left(\sin (4 \Lambda )-2 \cos (2 \Lambda ) \sin (3 c)-3 \sin (2 \Lambda ) \cos (3 c)-12
   \left((h_{12}^3)^2+(h_{13}^3)^2\right) \cot \left(\Lambda +\frac{3 c}{2}\right)\right)=0.
$$ 
We must split again into two cases.\\
\emph{Case 2.1.}\ $h_{13}^3\neq0.$
We have, of course, that
$$
\sin (4 \Lambda )-2 \cos (2 \Lambda ) \sin (3 c)-3 \sin (2 \Lambda ) \cos (3 c)-12
   \left((h_{12}^3)^2+(h_{13}^3)^2\right) \cot \left(\Lambda +\frac{3 c}{2}\right)=0
$$ and by \eqref{zero}, we may write 
\begin{multline}
\left(\cos (4 \Lambda
   )-2 \cos (2 \Lambda ) \cos (3 c)+12 (h_{12}^3)^2+12 (h_{13}^3)^2+1\right)\cot\left(\frac{3c}{2}-\Lambda \right)-\\
-\left(\sin (4 \Lambda )-2 \cos (2 \Lambda ) \sin (3 c)-3 \sin (2 \Lambda ) \cos (3 c)-12
   \left((h_{12}^3)^2+(h_{13}^3)^2\right) \cot \left(\Lambda +\frac{3 c}{2}\right)\right)=0.
\end{multline}
The latter equation reduces to $-3\cos(3c)\sin(2\Lambda)=0$, which implies $\cos(3c)=0$.
With this information, we evaluate \eqref{codazzi} for $E_1,E_2,E_1$ and, looking at the component of $E_2$ of the resulting vector gives
$$
\sin (3 c) \csc \left(\frac{3 c}{2}-\Lambda \right) \csc \left(\Lambda +\frac{3 c}{2}\right) \left(\cos (4 \Lambda
   )-2 \cos (2 \Lambda ) \cos (3 c)+12 (h_{12}^3)^2+12 (h_{13}^3)^2+1\right)=0.
$$
This yields $\cos (4 \Lambda )+12 (h_{12}^3)^2+12 (h_{13}^3)^2+1=0,$
which is a contradiction, as, given that $\Lambda$ is not constant,  the expression is actually strictly greater than $0$.\\

\emph{Case 2.2}\ $h_{13}^3=0.$
From \eqref{codazzi} evaluated for $E_1,E_2,E_2$; $E_1,E_3,E_3$; $E_2,E_3,E_3$; $E_3,E_2,E_2$, by looking at the components of  $E_2,E_3; E_3, E_2;E_3;E_3$, we obtain, respectively: 
\begin{align}
E_1(h_{23}^3)=&-h_{12}^3 h_{22}^3 \cot \left(\Lambda +\frac{3 c}{2}\right)+h_{12}^3 h_{22}^3 \cot (2 \Lambda
   )-\frac{h_{22}^3}{\sqrt{3}},\nonumber\\
E_2(h_{12}^3)=&-h_{12}^3 h_{23}^3 \left(-2 \cot (2 \Lambda )+\cot \left(\frac{3 c}{2}-\Lambda \right)+\cot \left(\Lambda +\frac{3
   c}{2}\right)\right),\nonumber\\
0=&-\frac{1}{3} \sin (2 \Lambda +3c)-2 (h_{12}^3)^2
   \left(\cot (2 \Lambda )+\cot \left(\frac{3
  c}{2}-\Lambda
   \right)\right)+\frac{h_{12}^3}{\sqrt{3}},\label{x}\\
E_1(h_{22}^3)=&\frac{1}{3} h_{23}^3 \left(\sqrt{3}-3 h_{12}^3 \left(\cot (2 \Lambda )+\cot \left(\frac{3 c}{2}-\Lambda \right)\right)\right),\nonumber\\
E_3(h_{12}^3)=&-h_{12}^3 h_{22}^3 \left(2 \cot (2 \Lambda )+\cot \left(\frac{3 c}{2}-\Lambda \right)+\cot \left(\Lambda +\frac{3
   c}{2}\right)\right),\nonumber\\
E_2(h_{22}^3) =&-E_3(h_{23}^3),\nonumber\\
E_3(h_{22}^3)=&\frac{1}{3} \Big(-\sin (4 \Lambda )-6 (h_{12}^3)^2 \left(\cot \left(\Lambda +\frac{3 c}{2}\right)-\cot \left(\frac{3
   c}{2}-\Lambda \right)\right)+3 E_2(h_{23}^3)-\sqrt{3} h_{12}^3- \Big.\nonumber \\ 
&\Big.-9 \cot (2 \Lambda )\left((h_{22}^3)^2+(h_{23}^3)^2\right)\Big).\nonumber
\end{align}
Next, for the vector fields $E_1,E_2,E_1,E_2$, we may evaluate the sectional  curvature  once using the definition for the curvature tensor,  once using \eqref{Gauss2},  and subtract the results. This gives
\begin{multline}\label{e69}
-\sin (2c) \sin (c-2
   \Lambda )+\cos (2 c) \cos (c-2 \Lambda )-\\
-6  (h_{12}^3)^2 \csc (2 \Lambda ) \cos \left(\frac{3
   c}{2}-\Lambda \right) \csc \left(\Lambda +\frac{3
   c}{2}\right)+\sqrt{3} h_{12}^3 \cot
   \left(\frac{3 c}{2}-\Lambda \right)+1=0.
\end{multline}
From \eqref{x}, we obtain 
\begin{equation}\label{xx}
(h_{12}^3)^2=\frac{\sqrt{3} h_{12}^3-\sin (2 \Lambda +3 c)}{6
   \left(\cot (2 \Lambda )+\cot \left(\frac{3
   c}{2}-\Lambda \right)\right)},
\end{equation} so that  we may replace $(h_{12}^3)^2$ in \eqref{e69} and solve for $h_{12}^3$:
\begin{equation} \label{xxx}
h_{12}^3=\frac{(\cos(3c)-\cos(2\Lambda))\csc(2\Lambda)}{\sqrt{3}}.
\end{equation}
Nevertheless, $(h_{12}^3)^2$ from \eqref{xxx}  does not coincide with \eqref{xx}, as it would imply 
$$\csc ^2(2 \Lambda ) (\cos (3 c)-\cos (2 \Lambda )) (-9
   \cos (2 \Lambda )+\cos (6 \Lambda )+8 \cos (3 c))=0,$$
i.e. $\Lambda$ should be constant, which is a contradiction.
\end{proof}

A complete  classification of the Lagrangian submanifolds with $\theta_1=\frac{\pi}{3}$ is given in \cite{eu}. Similarly, for those with angle functions $\theta_1=0$ or $\theta_1= \frac{2\pi}{3}$, we obtain the same result  by constructions $\tilde{f}$ and $f^*$, respectively. We recall that the following theorems are proven in \cite{eu}:

\begin{theorem}\label{t1}
Let $\omega$ and $\mu$ be solutions of, respectively, the Sinh-Gordon equation $\Delta\omega=-8\sinh\omega$ and the Liouville equation  $\Delta\mu=-e^\mu$ on an open simply connected domain $U \subseteq\mathbb{C}$ and let $p:U\rightarrow\mathbb{S}^3$ be the associated minimal surface with complex coordinate $z$ such that $\sigma(\partial z,\partial z)=-1.$\\
Let $V=\{(z,t)\mid z\in U, t\in \mathbb{R}, e^{\omega+\mu}-2-2\cos(4t)>0 \}$ and let $\Lambda$ be a solution of $$\left(\frac{2\sqrt{3}e^\omega}{\tan\Lambda}-2\sin(2t) \right)= e^{\omega+\mu}-2-2\cos(4t)$$ on $V$. Then, there exists a Lagrangian immersion $f:V\rightarrow \nks : x\mapsto (p(x),q(x))$, where $q$ is determined by 
\begin{align*}
\frac{\partial q}{\partial t}=& -\frac{\sqrt{3} }{2 \sqrt{3} e^{\omega }-2 \sin
   (2t) \tan \Lambda }\ q\ \alpha_2 \times \alpha_3, \\
\frac{\partial q}{\partial u}=& \frac{1}{8} \Bigl(e^{-\omega } \left( {\mu_v}+ {\omega_v}-\frac{( {\mu_u}+ {\omega_u}) \cos
   (2 t) \tan \Lambda }{\sqrt{3} e^{\omega }
  -\sin (2 t) \tan \Lambda }\right)q \ \alpha_2 \times \alpha_3  - 4 (\sqrt{3} \cot \Lambda \cos (2 t)+1)\ q\  \alpha_2 - \Bigr. \\ 
& \Bigl. -4 \sqrt{3}\sin (2 t) \cot\Lambda \ q\ \alpha_3   \Bigr), \\
\frac{\partial q}{\partial v}=& \frac{1}{8} \Bigl(-e^{-\omega } \left( {\mu_u}+ {\omega_u}+\frac{( {\mu_v}+ {\omega_v}) \cos
   (2 t) \tan\Lambda }{\sqrt{3} e^{\omega }
  -\sin (2 t) \tan \Lambda }\right) \ q\ \alpha_2 \times \alpha_3  - 4 \sqrt{3} \cot \Lambda \sin (2 t)\  q\ \alpha_2+ \Bigr. \\ &\Bigl.  +4(1+ \sqrt{3}\cos (2 t) \cot\Lambda)\ q \ \alpha_3   \Big),
\end{align*}
where $\alpha_2=\bar p p_u$ and $\alpha_3=\bar p p_v$.
\end{theorem}

\begin{theorem}\label{t2}
Let $X_1,X_2,X_3$ be the standard vector fields on $\mathbb{S}^3$. Let $\beta$ be a solution of the differential equations 
\begin{align*}
&X_1(\beta)=0,\\
&X_2(X_2(\beta))+X_3(X_3(\beta))=\frac{2(3-e^{4\beta})}{e^{4\beta}},
\end{align*}
on a connected, simply connected open subset $U$ of $\mathbb{S}^3$.\\
 Then there exist a Lagrangian immersion $f:U\rightarrow \nks : x\mapsto (p(x),q(x))$, where $p(x)=xix^{-1}$ and $q$ is determined by
\begin{align*}
%\left\{
\begin{array}{l}
X_1(q)=-2qhxix^{-1}h^{-1},\\
X_2(q)=q\left(-X_3(\beta)  hxix^{-1}h^{-1}-(1-\sqrt{3} e^{-2\beta})\ hxjx^{-1}h^{-1}\right),\\
X_3(q)=q\left( X_2(\beta)\ hxix^{-1}h^{-1} -(1+\sqrt{3} e^{-2\beta}) \ hxkx^{-1}h^{-1}\right).
\end{array}
%\right.
\end{align*}
\end{theorem}

%Note that in the previous theorem the image of $p$ is a totally geodesic surface in $\mathbb{S}^3.$

\begin{theorem}\label{t3}
Let $\omega$ be a solution of the Sinh-Gordon equation $\Delta\omega=-8\sinh\omega$ on an open connected domain of $U$ in $\mathbb{C}$ and let $p:U\rightarrow \mathbb{S}^3$ be the associated minimal surface with complex coordinate $z$ such that $\sigma(\partial z,\partial z)=-1.$ Then, there exist a Lagrangian immersion $f:U\times \mathbb{R}\rightarrow\nks:x\mapsto (p(x),q(x)) $, where $q$ is determined by 
\begin{align*}
&\frac{\partial q}{\partial t}=-\frac{\sqrt{3}e^{-\omega}}{4}q \ \alpha_2\times\alpha_3,\\
&\frac{\partial q}{\partial u}=\frac{e^{-\omega}}{8}(4e^{\omega}q \alpha_2-4  q \alpha_3+\omega_v q\ \alpha_2\times\alpha_3),\\
&\frac{\partial q}{\partial v}=-\frac{e^{-\omega}}{8}(4 q \alpha_2-4e^{\omega}q \alpha_3+\omega_u  q\ \alpha_2\times\alpha_3).
\end{align*}
where $\alpha_2=\bar p p_u$ and $\alpha_3=\bar p p_v.$
\end{theorem}

\begin{theorem}
Let $f:M\rightarrow \nks : x\mapsto (p(x),q(x))$ be a Lagrangian immersion such that $p$ has nowhere maximal rank. Then every point $x$ of an open dense subset of $M$ has a neighborhood $U$ such that $f|_U$ is obtained as described in Theorem \ref{t1}, \ref{t2} or \ref{t3}.
\end{theorem}

\end{document}